\documentclass[a4paper,12pt]{article}

\usepackage{fancyhdr}

\usepackage[scale=0.7]{geometry}

\usepackage{amsmath}
\usepackage{amsthm}
\usepackage{amssymb}
\usepackage{amscd}          			

\usepackage{bbm}            			
\usepackage{mathrsfs}           		

\usepackage{graphicx}          		

\usepackage{verbatim}           		

\theoremstyle{plain}
\newtheorem{Theorem}{Theorem}[section]
\newtheorem{Lemma}[Theorem]{Lemma}

\newtheorem{Proposition}[Theorem]{Proposition}
\newtheorem{Conjecture}[Theorem]{Conjecture}

\theoremstyle{definition}
\newtheorem{Definition}[Theorem]{Definition}
\newtheorem{Example}[Theorem]{Example}

\theoremstyle{remark}
\newtheorem{Remark}[Theorem]{Remark}

\newcommand{\N}{\mathbb{N}}     
\newcommand{\R}{\mathbb{R}}     

\newcommand{\Z}{\mathbb{Z}}         

\newcommand{\calB}{\mathscr{B}}

\newcommand{\calF}{\mathscr{F}}
\newcommand{\calG}{\mathscr{G}}

\newcommand{\calL}{\mathscr{L}}

\newcommand{\calR}{\mathscr{R}}

\newcommand{\calX}{\mathscr{X}}

\newcommand{\bdelta}{\boldsymbol{\delta}}

\newcommand{\bR}{\boldsymbol{R}}

\newcommand{\bT}{\boldsymbol{T}}

\newcommand{\bb}{\boldsymbol{b}}

\DeclareMathOperator{\diam}{diam}			

\usepackage[mathcal]{eucal}					

\newcommand{\hel}
{
\hskip2.5pt{\vrule height8pt width.5pt depth0pt}
\hskip-.2pt\vbox{\hrule height.5pt width8pt depth0pt}
\,
}

\renewcommand{\geq}{\geqslant}
\renewcommand{\leq}{\leqslant}
\renewcommand{\epsilon}{\varepsilon}

\newcommand{\limsupe}{\mathop{\overline{\lim}}}

\usepackage{color}

\begin{document}

\title{Almost everywhere convergence for Lebesgue differentiation {processes} along rectangles}

\author{E.~D'Aniello, A.~Gauvan, L.~Moonens and J.~Rosenblatt}
\date{June, 2022}

\maketitle

\begin{abstract}
In this paper, we study Lebesgue differentiation processes along rectangles $R_k$ shrinking to the origin in the Euclidean plane, and the question of their almost everywhere convergence in $L^p$ spaces. In particular, classes of examples of such processes failing to converge a.e. in $L^\infty$ are provided, for which $R_k$ is known to be oriented along the slope $k^{-s}$ for $s>0$, yielding an interesting counterpart to the fact that the directional maximal operator associated to the set $\{k^{-s}:k\in\N^*\}$ fails to be bounded in $L^p$ for any $1\leq p<\infty$.
\end{abstract}


\maketitle

\section{Introduction}

Given a measure space $(\Omega,\calF,\mu)$ and a sequence $\bT:=(T_k)_{k\in\N}$ of linear operators $T_k:L^p(\Omega,\calF,\mu)\to L^0(\Omega,\calF,\mu)$ sending a given $f\in L^p(\Omega,\calF,\mu)$ to a measurable function $T_kf$ for each $k\in\N$, the question of the \emph{almost everywhere convergence} of $T_kf$ to $f$ \emph{for all} $f\in L^p(\Omega,\calF,\mu)$ is an important problem in real analysis, especially in the case where $(T_k)$ approximates the identity map in $L^p(\Omega,\calF,\mu)$ when $k$ grows. Under mild hypotheses on the operators $T_k$, $k\in\N$ (like continuity in measure and, for some results, positivity and their commuting with translations), it follows from general principles, associated to the names of E.~Stein, D.~Burkholder and S.~Sawyer, that the holding of the above a.e. convergence property in $L^p(\Omega, \calF,\mu)$ is equivalent to the fact that the associated maximal operator $T^*$ defined for a given function $f\in L^p(\Omega,\calF,\mu)$ by:
$$
T^*f(x):=\sup_{k\in\N} |T_kf(x)|,
$$
either satisfies an inequality of weak type $(p,p)$ (see below for a precise definition) or satisfies the seemingly weaker property that one has $T^*f<+\infty$, $\mu$-a.e. in $\Omega$, for all $f\in L^p(\Omega,f,\mu)$ (see \emph{e.g.} Stein \cite{STEINART}, Burkholder \cite{BURKHOLDER} and Sawyer \cite{SAWYER} for original formulations of this principle, and Garsia \cite{GARSIA} for a beautiful exposition of this way of dealing with a.e. convergence of processes).\\

In the current paper, we will be mainly interested in \emph{Lebesgue differentiation processes} $\bT=\bT_{\bR}$ on $L^p(\R^2)$ associated to sequences $\bR=(R_k)$ of rectangles in the plane centered at the origin, the diameters of which tend to zero, \emph{i.e.} in the case where one has, for each $k\in\N$:
$$
T_kf(x):=\frac{\chi_{R_k}}{|R_k|}*f(x)=\frac{1}{|R_k|}\int_{x+R_k} f.
$$
Note that we shall here typically work in the case where the classical Lebesgue differentiation theorem does \emph{not} provide the a.e. convergence of $T_kf$ to $f$ for $f\in L^1(\R^2)$ or even in $L^p(\R^d)$ for $p>1$: one may remember, indeed, that the Lebesgue differentiation theorem provides a.e. convergence of $T_kf$ to $f$  for each $f\in L^1(\R^2)$ in this context if $\bR$ is a sequence of \emph{squares}, or in case the ratio between the two side-lengths of $R_k$, called its \emph{shape}, remains away (uniformly) from $0$ and $\infty$, and that it follows from Jessen, Marcinkiewicz and Zygmund \cite{JMZ} that this a.e. convergence holds for any $f\in L\log L(\R^2)$ (and hence also in $L^p(\R^2)$ for any $p>1$) if $\bR$ is a sequence of rectangles parallel to the coordinate axes (see \cite{DM17} and, for example, the introduction in \cite{DMZAA} for a review of the different versions one can formulate of the Lebesgue differentiation with rectangles).

In the following paper, we will hence be mainly dealing with sequences $\bR$ of rectangles whose shapes are tending to $0$ (or $\infty$) and which will usually \emph{not} be parallel to the coordinate axes.\\

As we mentioned briefly above, and as it will be made precise below, the a.e. convergence of $(T_k f)$ to $f$ for all $f\in L^p(\R^2)$ ($1\leq p<\infty$) is equivalent to having $T^*f<+\infty$ almost everywhere in $\R^2$ for all $f\in L^p(\R^2)$, or the existence of a constant $C>0$ such that for any $f\in L^p(\R^2)$ and any $\lambda>0$, one has:
$$
|\{x\in\R^2 : T^*f(x)>\lambda\}|\leq \frac{C}{\lambda^p}\|f\|_p^p.
$$
If the maximal operator $T^*$ satisfies such an inequality (called weak type $(p,p)$), we shall say that it is \emph{$L^p$-good}; in this case we shall also say that the sequence $\bR$ is $L^p$-good, keeping in mind that it is equivalent to the a.e. convergence of $T_kf$ to $f$ for all $f\in L^p(\R^2)$. If $T^*$ (or $\bR$) is not $L^p$-good, we shall call it \emph{$L^p$-bad}.

It also follows from Hagelstein and Parissis \cite{HP} that the a.e. convergence of $(T_kf)$ to $f$ for all $f\in L^\infty(\R^2)$ is equivalent to the following property: for each $0<\lambda <1$, there exists a constant $C>0$ such that for any Borel set $B\subseteq\R^2$ with finite Lebesgue measure, one has:
\begin{equation}\label{eq.linfini-good}
|\{x\in\R^2: T^*\chi_B(x)>\lambda\}|\leq C_\lambda |B|.
\end{equation}
As before, if the above inequality is satisfied, we shall say that $T^*$ and $\bR$ are \emph{$L^\infty$-good}, meaning in particular that $T_kf$ converges a.e. to $f$ for all $f\in L^\infty(\R^2)$. Again, if $T^*$ or $\bR$ is not $L^\infty$-good, we shall say it is \emph{$L^\infty$-bad}.\\

Here an interesting comparison with \emph{directional maximal operators} is relevant. Given a set $\Omega\subseteq [0,+\infty)$ (thought of as a set of \emph{slopes}), denote by $M_\Omega$ the maximal operator defined for a given function $f$ on $\R^2$ by:
$$
M_\Omega f(x):=\sup \frac{1}{|R|}\int_R |f|,
$$
where the upper bound is extended on all rectangles containing $x$, \emph{one side of which has a slope $\omega\in\Omega$}. Such a maximal operator is called \emph{directional}.

A deep and beautiful result by Bateman \cite{BATEMAN} states a fundamental dichotomy for directional maximal operators; it can be stated in the following way (we here extend the meaning of ``$L^p$-good'' and ``$L^p$-bad'' to the maximal operator $M_\Omega$ in an obvious way, even though it is not associated to a process):
\begin{enumerate}
\item[(i)] either $M_\Omega$ is $L^p$-good for any $1<p\leq\infty$ (in which case we call $\Omega$ a \emph{good set of directions});
\item[(ii)] or $M_\Omega$ is $L^p$-bad for any $1<p\leq\infty$ (in which case we call $\Omega$ a \emph{bad set of directions}).
\end{enumerate}
Note that Bateman's original dichotomy was stated for finite $1<p<\infty$; the observation that $p=\infty$ could be included was made by three of the current paper's authors in \cite{DMR}. Moreover, Bateman gives a geometric characterization of good sets of directions; we refer to his work \cite{BATEMAN} for more details. Yet, we can say here, for example, that geometric sequences like $\Omega:=\{2^{-k}:k\in\N\}$ are examples of good sets of directions, while sets as simple as $\Omega_s:=\{k^{-s}:k\in\N^*\}$ for $s>0$, are examples of bad sets of directions.\\

In the above setting where $\bT=\bT_{\bR}$ is the sequence of Lebesgue averages associated to a sequence $\bR=(R_k)$ of rectangles centered at the origin, the diameters of which tend to zero, we can already formulate an immediate consequence of Bateman's result. To this purpose, denote for each $k\in\N$ by $\omega_{R_k}\in [0,+\infty)$ the slope of the longest side of $R_k$ (in case $R_k$ is a square, denote by $\omega_{R_k}$ the minimum of the two slopes of its sides) and let $\Omega_{\bR}:=\{\omega_{R_k}:k\in\N\}$ the set of those slopes.

If $\Omega_{\bR}$ is a good set of directions (in the sense of Bateman recalled above), then we obviously have, using the same notations as before:
\begin{equation}\label{eq.intro1}
T^*\leq M_{\Omega_{\bR}},
\end{equation}
from which it immediately follows that $T^*$ and $\bR$ are $L^p$-good for any $1<p\leq\infty$. It is then also the case that $T_kf$ converges a.e. to $f$ for every $f\in L^p(\R^2)$. In this case, the problem we stated initially hence has a trivial (positive) solution.

Yet in case $\Omega_{\bR}$ is a \emph{bad} set of directions, inequality \eqref{eq.intro1} does not provide any information on the good or bad character of $T^*$; roughly speaking, the process $T^*$ ``extracts'' from $M_{\Omega_R}$ the smallest possible amount of rectangles still providing a differentiation scheme, but may not capture anymore the geometry of rectangles that makes $M_{\Omega_R}$ bad.

It is precisely the main focus of our paper to provide an answer to the following (rather vague) question: assuming $\Omega$ is a bad set of directions like $\Omega_s:=\{k^{-s}:k\in\N^*\}$ for $s>0$, can one still provide examples of sequences $\bR$ of rectangles as above satisfying $\Omega_{\bR}=\Omega$ and for which the Lebesgue differentiation process $\bT_{\bR}$ is still $L^\infty$-bad?

As we shall explain in the next section, presenting our paper's results, the answer to the latter is positive. On the way to answering it, we shall also present other results concerning $L^1$-good and $L^p$-good processes associated to sequences of rectangles of the above type.\\

A last remark should be made concerning the fact that we insist on working with differentiation \emph{processes} here, in opposition to differentiation \emph{bases} (see de Guzmán \cite{GUZMAN} for the terminology and the precise definitions, on which we do not want to insist here): while to any differentiation process of the type $\bT_{\bR}$ for some sequence of rectangles $\bR$, corresponds a centered, translation-invariant differentiation basis $$
\calB_{\bR}=\bigcup_{x\in\R^2}\calB_{\bR}(x):=\bigcup_{x\in\R^2}\{x+R_k: k\in\N\},
$$
one should insist on the fact that a differentiation basis $\calB$ can, at a given point $x\in\R^2$, have an uncountable class of admissible sets $\calB(x)$, even more so when $\calB$ is a Busemann-Feller-type basis (\emph{i.e.} satisfies that, given $B\in\calB$, one has $B\in\calB(x)$ if and only if $B\ni x$).

Observe that if $\Omega$ is a set of directions in the plane (see above), the set $\calB_\Omega$ of all rectangles oriented along one direction $\omega\in \Omega$ is a translation-invariant Busemann-Feller basis called a \emph{directional basis}.\\

When a translation-invariant differentiation basis $\calB=\bigcup_{x\in\R^d}\calB(x)$ in $\R^d$ consists of convex sets \emph{and is known to differentiate $L^\infty(\R^2)$} (in which case it is called a \emph{density basis}), it was shown by G.~Oniani (see \cite[Remark~7]{ONIANI}) that for any $1\leq p<\infty$, $\calB$ differentiates $L^1(\R^d)\cap L^p(\R^d)$ if and only if $\calB^*$ does, where $\calB^*=\bigcup_{x\in\R^d}\calB^*(x)$ is the Busemann-Feller basis associated to $\calB$ defined for $x\in\R^d$ by $\calB^*(x):=\{B\in\calB:B\ni x\}$.\\
 
When $\calB$ fails to differentiate $L^\infty(\R^d)$ (as is the case when $d=2$ for directional bases $\calB_\Omega$ associated to a bad sequence of directions $\Omega=\{\omega_k:k\in\N\}$), the latter equivalence is \emph{not} true anymore (see Hagelstein and Parissis \cite{HP}); our class of examples will show that, under some conditions on the sequence $(\omega_k)$, one can nevertheless construct a sequence $\bR:=(R_k)$ of rectangles $R_k$ oriented along direction $\omega_k$, for which the associated process $T_{\bR}$ is $L^\infty$-bad, hence extracting ``ordered'' bad \emph{processes} from the bad directional basis $\calB_\Omega$.\\

Let us now formulate our results in a more precise way.

\section{Results}

In this whole section, we use the same notations as in the introduction, and associate to any sequence $\bR=(R_k)$ of rectangles in the plane centered at the origin, the diameters of which tend to zero, a Lebesgue differentiation process $\bT=\bT_{\bR}=(T_k)$.

A first observation in the following paper (to which Section~\ref{sec.L1} is devoted) will be to provide a geometrical condition on $\bR$ ensuring that $\bT_{\bR}$ is $L^1$-good. More precisely, we shall prove the following result.

\begin{Theorem}\label{THM1}
If there exists $\alpha>0$ such that for all $k\in\N$ and all $l\in\N$ satisfying $l>k$, one has $$R_k-R_l\subseteq \alpha R_k$$ then the process $\bT$ is $L^1$-good.
\end{Theorem}

The latter geometric condition appearing in Theorem~\ref{THM1} should be thought of as a kind of \emph{nesting} property of the sequence $(R_k)$. The proof of Theorem~\ref{THM1} will rely on a geometric interpretation of the notion of \textit{correct factors} $Q_k$, $k\in\N$ of the sequence $\bR$ introduced by Rosenblatt and Wierdl in \cite{RW} and defined for $k\in\N$ by:
$$
Q_k:=\left|\bigcup_{l=k}^\infty R_k-R_l\right|.
$$
The most important property enjoyed by the correct factors is the following: if one has $Q_k<+\infty$ for any $k\in\N$, then the \emph{corrected maximal operator} $\tilde{T}^*$ defined for a given locally integrable $f$ and $x\in\R^2$ by:
$$
\tilde{T}^*f(x):=\max_{k\in\N} \left| \frac{|R_k|}{Q_k} T_kf(x)\right|=\max_{k\in\N}  \left|\frac{1}{Q_k}\int_{x+R_k} f\right|,
$$
always has weak type $(1,1)$. It is that property, combined with the geometric hypothesis made on $(R_k)$ in the statement of Theorem~\ref{THM1}, that will ensure $\bT_{\bR}$ to be $L^1$-good.\\

In section~\ref{sec.Lp}, we provide a continuous analogous result to the discrete ``corrected'' weak-type $(p,p)$ suggested by Rosenblatt and Wierdl in \cite[p.~551]{RW}. More precisely, we show the following result.
\begin{Theorem}\label{thm.Lp}
Fix $1\leq p<\infty$ and assume that one has $Q_k<\infty$ for each $k\in\N$. Define, for each $k\in\N$, the $k^{\text{th}}$ $L^p$-correct factor (of the sequence $\bR$) $W_{k,p}$ by letting $W_{k,1}:=Q_k$ and, in case $1<p<\infty$, by letting:
$$
W_{k,p}:=(Q_k)^{\frac 1p }|I_k|^{\frac{1}{p'}},
$$
where $1<p'<\infty$ is the conjugate exponent to $p$ satisfying $\frac 1p + \frac{1}{p'}=1$. Under those assumptions, the corrected maximal operator $\tilde{T}_p^*$ defined for any locally integrable $f$ and any $x\in\R^2$ by:
$$
T_p^*f(x):=\max_{k\in\N} \left|\frac{|R_k|}{W_{k,p}} T_kf(x)\right|=\max_{k\in\N}  \left|\frac{1}{W_{k,p}}\int_{x+R_k} f\right|,
$$
has weak type $(p,p)$.
\end{Theorem}
It will also be observed that it is not obvious, in case $p>1$, to get a sufficient geometric condition on $\bR$ ensuring that $\bT_{\bR}$ is a $L^p$-good process, using the latter theorem.\\

In sections~\ref{sec.Linfini1} and \ref{sec.Linfini2}, we will fix an nondecreasing sequence $\bb=(b_k)$ of positive real numbers satisfying $b_0=0$ and define points $B_k:=(b_k,0)$ for all $k\in\N$ and let $A:=(0,1)$; for each $k\in\N^*$ we then denote by $\Delta_k$ the triangle $AB_{k-1}B_{k}$. For each $k\in\N^*$, we associate to $\Delta_k$ a rectangle $P_k$ centered at the origin, obtained by translating a rectangle oriented along $AB_{k}$ and containing the image of $\Delta_k$ by a homothecy based at $A$ with ratio $\frac 32$ ({see Figure~\ref{fig.Delta_k} for an overview of the situation}, and section~\ref{sec.Linfini2} for a precise description). Finally we fix any sequence $\boldsymbol{\delta}=(\delta_k)_{k\in\N^*}$ of positive real numbers and we consider the sequence of rectangles $\bR=\bR(\bb,\bdelta)=(R_k)_{k\in\N^*}$ defined by $R_k:=\delta_n P_k$ if $k\in\N^*$ satisfies $2^n\leq k<2^{n+1}$ for a given $n\in\N$; we call $\bdelta$ an \emph{admissible sequence} in case one has $\diam R_k\to 0$, $k\to\infty$.
\begin{figure}
\begin{center}
\includegraphics[width=12cm]{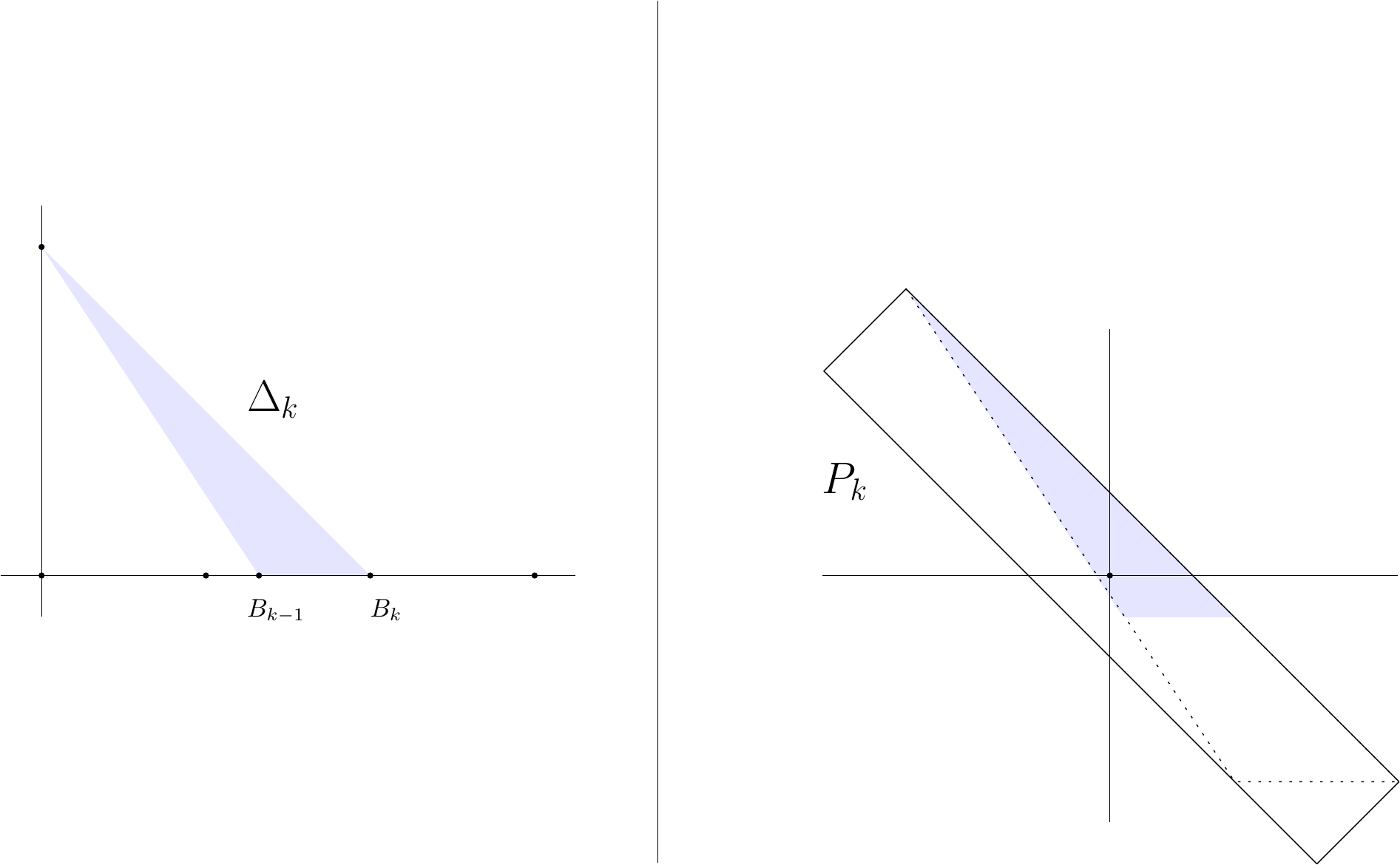}
\caption{The triangle $\Delta_k$ and the rectangle ${P}_k$}\label{fig.Delta_k}
\end{center}
\end{figure}

\begin{Theorem}\label{THM2}
Assume that the sequence $\boldsymbol{b}$ satisfies the following two conditions:
\begin{enumerate}
\item[(i)] there exists a constant $c>0$ such that one has $1 + b_{k-1}^2 \geq c(b_{k} -b_{k-1})^2$ for all $k\in\N^*$;
\item[(ii)] one has: $$G_{ \boldsymbol{b} }:=\sup_{n\in\N}\sup_{1\leq l\leq n} \left(\frac{b_{n+2l}-b_{n+l}}{b_{n+l}-b_n}+\frac{b_{n+l}-b_n}{b_{n+2l}-b_{n+l}}\right)<+\infty.$$
\end{enumerate}
Under those assumptions, the process $\bT_{\bR}$ associated to $\bR = \bR(\bb,\boldsymbol{\delta})$ is $L^\infty$-bad for any admissible sequence $\bdelta$.
\end{Theorem}
\begin{Remark}
In the latter statement, the first condition is merely a technical one, while the finiteness of $G_{\bb}$ is a quantitative way to ensure that the set of directions $\Omega(\bb)$ associated to $\bR$ (being the set of slopes of the segments $AB_{k}$, $k\in\N^*$) defined by:
$$
\Omega(\bb):=\left\{\frac{-1}{b_k}:k\in \N^*\right\},
$$
is a \emph{bad set of directions} (see the introduction above for a precise definition of a \emph{bad} set of directions). The quantity $G_{\bb}$, which one can call the \emph{Perron factor} of $\bb$, was first introduced by Hare and Rönning \cite{HR}; it was also used by A.~Gauvan in his Master's thesis \cite{GAUVANMEMOIRE} in this precise context, and in \cite{GAUVANPT} for providing concrete examples of homothecy-invariant bases of rectangles differentiating all \emph{vs.} no $L^p$ spaces.
\end{Remark}
\begin{Example}
One can verify that for $s>0$, the sequence $\boldsymbol{b}_s = (k^s)$ satisfies the conditions of Theorem \ref{THM2}.
\end{Example}

Let us now detail how one can obtain the $L^1$-a.e. convergence result stated in Theorem~\ref{THM1}.

\section{Almost everywhere convergence in $L^1$}\label{sec.L1}

We keep the notations used before, and associate to any sequence $\bR=(R_k)$ of rectangles in the plane centered at the origin, the diameters of which tend to zero, a Lebesgue differentiation process $\bT=\bT_{\bR}=(T_k)$. Throughout this section we'll omit the index $\bR$, keeping in mind that we shall always be working with the differentiation process $\bT=\bT_{\bR}$ associated to $\bR$.

Recall that one associates to $\bR$ its \emph{correct factors} $Q_k$, $k\in\N$, defined by Rosenblatt and Wierdl in \cite{RW} by letting, for $k\in\N$:
\begin{equation}\label{eq.corr-fact}
Q_k:=\left|\bigcup_{l=k}^\infty R_k-R_l\right|.
\end{equation}
We define the \emph{corrected maximal function} $\tilde{T}^*$ by letting, for a locally integrable function $f$ and $x\in\R^2$
$$
\tilde{T}^*f(x):=\sup_{k\in\N} \left|\frac{1}{Q_k} \int_{x+I_k} f\right|.
$$
The following theorem is taken from Rosenblatt and Wierdl \cite[Theorem~5.11]{RW} and explains the term ``correct factor''.
\begin{Theorem}\label{thm.RW}
There exists $C>0$ such that for all $f\in L^1(\R^2)$ and all $\lambda>0$, one has:
$$
|\{x\in\R^2: \tilde{T}^*f(x)>\lambda\}|\leq\frac{C}{\lambda} \|f\|_1.
$$
\end{Theorem}

As an easy consequence, this yields a sufficient condition for a sequence of rectangles $\bR$ to be $L^1$-good. 
\begin{Theorem}\label{thm.cond-suff} If there exists a constant $C>0$ such that, for all $k\in\N$, one has $Q_k\leq C|R_k|$, then there exists $C'>0$ so that one has, for all $f$ and $\lambda>0$:
$$
|\{x\in\R^2: T^* f(x)>\lambda\}|\leq\frac{C'}{\lambda} \|f\|_1,
$$
where $T^*$ is the maximal operator associated to $\bT$.
\end{Theorem}
\begin{Remark}\label{rmq.thm-cond-suff}
It follows immediately from the Sawyer-Stein's principle (see above in the introduction) that, under the hypotheses of the latter theorem, $\bR$ yields an $L^1$-good process $\bT$.
\end{Remark}
\begin{proof}
By hypothesis, one has $\eta:=\inf_k \frac{|R_k|}{Q_k}>0$. The theorem then follows immediately from Rosenblatt and Wierdl's result (Theorem~\ref{thm.RW} above).
\end{proof}

We now have a closer look at what the hypothesis in the latter theorem (\emph{i.e.} the growth of the correct factor at a comparable speed to that of the rectangle's areas) means geometrically. To that purpose, a first simple geometrical observation will be useful.

\subsection*{Differences of symmetric rectangles}\label{sec.geom-obs}
Fix real numbers $0<\ell\leq L$, $0< \ell'\leq L'$, assume that one has $L'\leq L$, define a rectangle $R=[-L/2,L/2]\times [-\ell/2,\ell/2]$ parallel to the axes and denote by $R'$ the rectangle obtained by rotating the rectangle $[-L'/2,L'/2]\times [\ell'/2,\ell'/2]$ of an angle $\omega\neq 0$ around the origin. Denote then by $\widehat{R-R'}=[-\hat{L}/2,\hat{L}/2]\times[-\hat{\ell}/2,\hat{\ell}/2]$ the smallest rectangle parallel to the axes that contains the set $R-R'$ (see Figure~\ref{fig.R-R'}).

\begin{figure}[h]
\begin{center}
\includegraphics[scale=1.2]{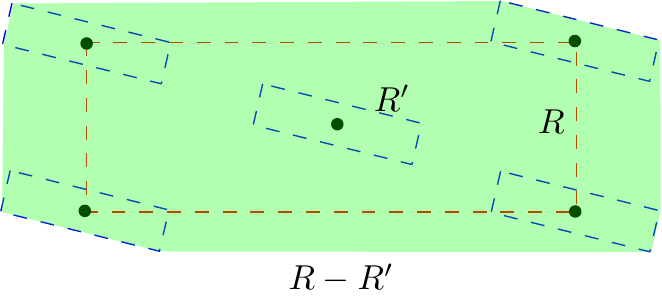}\caption{The rectangles $R$, $R'$ and the set $R-R'$}\label{fig.R-R'}
\end{center}
\end{figure}

It is easy to see, studying the coordinates of points in $(L/2,\ell/2)+R'$, that one has:
\begin{equation}\label{eq.R-R'x}
\hat{L}=L+\max_{\eta,\theta\in [-1,1]}(L'\eta\cos\omega-\ell'\theta\sin\omega)= L+L'|\cos\omega|+\ell'|\sin\omega|,
\end{equation}
and:
\begin{equation}\label{eq.R-R'y}
\hat{\ell}=\ell+\max_{\eta,\theta\in [-1,1]}(L'\eta\sin\omega+\ell'\theta\cos\omega)= \ell+L'|\sin\omega|+\ell'|\cos\omega|.
\end{equation}
Moreover, using the fact that $R-R'$ contains the two parallelograms $P_1$ and $P_2$ represented on Figure~\ref{fig.P1}, we also get:
\begin{equation}\label{eq.R-R'min}
|R-R'|\geq \max(\ell L'|\cos\omega|,LL' |\sin\omega|).
\end{equation}

\begin{figure}[h]
\begin{center}
\includegraphics[scale=0.9]{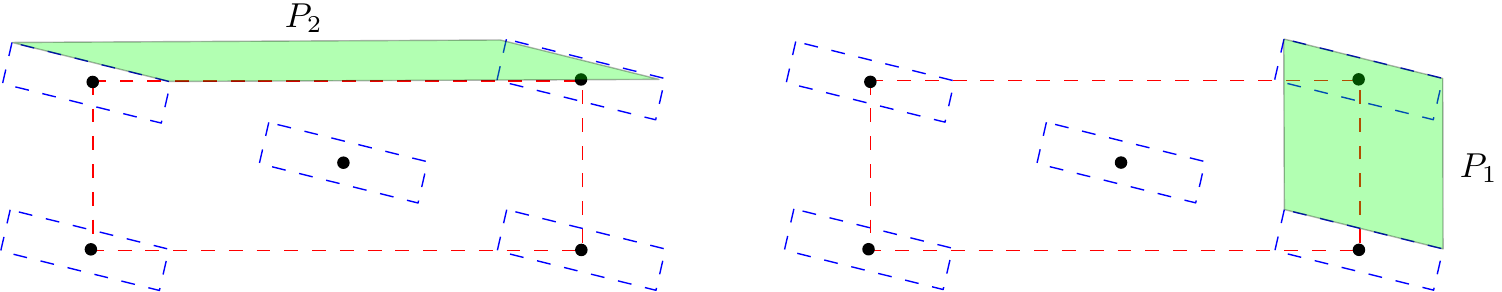}\caption{The parallelograms $P_1$ and $P_2$}\label{fig.P1}
\end{center}
\end{figure}

We are now ready to express in geometrical terms the growth condition on the correct factor appearing in the statement of Theorem~\ref{thm.cond-suff}.

\subsection*{Correct factors and linear growth}
The next lemma expresses is the announced equivalence between the linear growth of the correct factors of $\bR$ and a geometrical property on the sequence $\bR$ itself expressing its ``nested'' behavior.
\begin{Lemma}\label{lem.linear}
Assume that $\bR =(R_k)$ is as before.
The following two properties are equivalent:
\begin{itemize}
\item[(i)] there exists $C>0$ such that one has $Q_k\leq C|R_k|$, where $(Q_k)$ is the sequence of correct factors associated to $(R_k)$ as above;
\item[(ii)] there exists $\alpha>0$ such that for all $k\in\N$ and all $l\in\N$ satisfying $l>k$, one has $R_k-R_l\subseteq \alpha R_k$.
\end{itemize}
\end{Lemma}
\begin{Remark}
In the above statement, property (i) will be referred to by saying that the correct factors of $\bR$ \emph{have linear growth}; property (ii) will be expressed by saying that $\bR$ is \emph{almost nested}.
\end{Remark}
\begin{proof}
That (ii) implies (i) is obvious from the definition of $Q_k$ for $k\in\N$.

To prove that (i) implies (ii), start by choosing, for each $k\in\N$, 
positive real numbers $L_k$ and $\ell_k$ and angles $\theta_k\in (-\pi/2,\pi/2)$ in such a way that $R_k$ is obtained by rotating $[-L_k/2,L_k/2]\times [-\ell_k/2,\ell_k/2]$ around the origin by an angle $\theta_k$. Assume that (ii) does not hold, and hence that for all $\alpha>0$ there exists integers $k_\alpha\in\N$ and $l_\alpha>k_\alpha$ so that $R_{k_\alpha}-R_{l_\alpha}$ is not included in $\alpha R_{k_\alpha}$.

Fix now $\alpha>2$. If one has $\theta_{k_\alpha}=\theta_{l_\alpha}$ then we see that $R_{k_\alpha}-R_{l_\alpha}$ is a rectangle centered at the origin, parallel to $R_{k_\alpha}$ and of sides $L_{k_\alpha}+L_{l_\alpha}$ and $\ell_{k_\alpha}+\ell_{l_\alpha}$ respectively. Since $R_{k_\alpha}-R_{l_\alpha}$ is not included in $\alpha R_{k_\alpha}$ we get either $L_{l_\alpha}+\ell_{l_\alpha}>\alpha L_{k_\alpha}$ or $\ell_{l_\alpha}+\ell_{k_\alpha}>\alpha\ell_{k_\alpha}$, implying in both cases that:
$$
Q_{k_\alpha}\geq |R_{k_\alpha}-R_{l_\alpha}|=(L_{k_\alpha}+L_{l_\alpha})\cdot (\ell_{k_\alpha}+\ell_{l_\alpha})\geq \alpha L_{k_\alpha}\ell_{k_\alpha}=\alpha | R_{k_\alpha}|.
$$

If one has $\theta_{k_\alpha}\neq \theta_{l_\alpha}$, then, after applying a rotation of angle $-\theta_{k_\alpha}$ around the origin to $R_{k_\alpha}$ and $R_{l_\alpha}$, we are in the situation of section~\ref{sec.geom-obs} with $R=R_{k_\alpha}$, $R'=R_{l_\alpha}$ and $\delta=\theta_{l_\alpha}-\theta_{k_\alpha}$. Since $R_{k_\alpha}-R_{l_\alpha}$ is not contained in $\alpha R_{k_\alpha}$, and since $\widehat{R_{k_\alpha}-R_{l_\alpha}}$ is the smallest rectangle centered at the origin, homothetic to $R_{k_\alpha}$ and containing $R_{k_\alpha}-R_{l_\alpha}$, this implies that $\widehat{R_{k_\alpha}-R_{l_\alpha}}$ is not contained in $\alpha R_{k_\alpha}$ and hence by \eqref{eq.R-R'x} and \eqref{eq.R-R'y} that one has either:
$$
\alpha L_{k_\alpha}<L_{k_\alpha}+L_{l_\alpha}|\cos (\theta_{l_\alpha}-\theta_{k_\alpha})|+\ell_{l_\alpha}|\sin (\theta_{l_\alpha}-\theta_{k_\alpha})|,
$$
or:
$$
\alpha \ell_{k_\alpha}<\ell_{k_\alpha}+L_{l_\alpha}|\sin (\theta_{l_\alpha}-\theta_{k_\alpha})|+\ell_{l_\alpha}|\cos (\theta_{l_\alpha}-\theta_{k_\alpha})|;
$$
since one has $\ell_{l_\alpha}\leq L_{\ell_\alpha}\leq L_{k_\alpha}$, those inequalities imply respectively:
$$
L_{l_\alpha}|\cos (\theta_{l_\alpha}-\theta_{k_\alpha})|\geq (\alpha-2) L_{k_\alpha},
$$
or:
$$
L_{l_\alpha}|\sin (\theta_{l_\alpha}-\theta_{k_\alpha})|\geq (\alpha-2) \ell_{k_\alpha}.
$$
In case the first inequality holds, \eqref{eq.R-R'min} implies that one has:
$$
Q_{k_\alpha}\geq |R_{k_\alpha}-R_{l_\alpha}|\geq \ell_{k_\alpha} L_{l_\alpha} |\cos (\theta_{l_\alpha}-\theta_{k_\alpha})|\geq (\alpha-2) L_{k_\alpha}\ell_{k_\alpha}=(\alpha-2)|R_{k_\alpha}|.
$$
In case it is the second of the above inequalities that holds, we get again using \eqref{eq.R-R'min}:
$$
Q_{k_\alpha}\geq |R_{k_\alpha}-R_{l_\alpha}|\geq L_{k_\alpha} L_{l_\alpha} |\sin (\theta_{l_\alpha}-\theta_{k_\alpha})|\geq (\alpha-2) L_{k_\alpha}\ell_{k_\alpha}=(\alpha-2)|R_{k_\alpha}|.
$$
We hence get $Q_{k_\alpha}\geq (\alpha-2)|R_{k_\alpha}|$ in all cases, contradicting the linear growth of the correct factors, that is property (i), since $\alpha$ can be arbitrarily large. The proof is hence complete.
\end{proof}

We are now ready to (re-)state and prove Theorem~\ref{THM1}.
\begin{Theorem}\label{THM1-v2}
Assume that $\bR=(R_k)$ is a sequence of rectangles in $\R^2$ centered at the origin, the diameters of which tend to zero.
If $\bR$ is almost nested, \emph{i.e.} if there exists $\alpha>0$ such that for all $k\in\N$ and all $l\in\N$ satisfying $l>k$, one has $$R_k-R_l\subseteq \alpha R_k$$ then the process $\bT=\bT_{\bR}$ is $L^1$-good.
\end{Theorem}
\begin{proof}
If $\bR$ is almost nested, then it follows from Lemma~\ref{lem.linear} that its correct factors $Q_k$, $k\in\N$, have linear growth. But it then follows from Remark~\ref{rmq.thm-cond-suff} that $\bT_{\bR}$ is $L^1$-good, as we wanted to show.
\end{proof}

For rectangles parallel to the coordinate axes, the correct factors allow us to rephrase a result by A.~Stokolos \cite{STOKOLOS88}.
\begin{Theorem}[Stokolos]\label{thm.stokolos} Assume that $\bR = (R_k)$ is as in Lemma~\ref{lem.linear} \emph{and that moreover all rectangles $R_k$ are parallel to the coordinate axes}. Then the following properties are equivalent:
\begin{itemize}
\item[(i)] $\bT$ is $L^1$-good;
\item[(ii)] $\bR$ can be decomposed into finitely many subsequences along which the correct factor has linear growth;
\item[(iii)] $\bR$ can be decomposed into finitely many subsequences which are almost nested.
\end{itemize}
\end{Theorem}
\begin{proof}
That (ii) and (iii) are equivalent is an immediate consequence of Lemma~\ref{lem.linear}. Now that (ii) implies (i) follows easily from Theorem~\ref{thm.cond-suff} and from Sawyer-Stein's principle. Finally, if $(R_k)$ is $L^1$-good, then it follows from a result by A.~Stokolos \cite{STOKOLOS88} (not formulated in this exact way, though: see Moonens and Rosenblatt \cite{MR} where it is explained how Stokolos' 1988 theorem can be rephrased in this fashion~---~or see the more general result in A.~Stokolos' survey \cite[Corollary~1, p.~1448]{STOKOLOS2005}), that the sequence $(R_k^*)$ of all dyadic enlargements of rectangles in $(R_k)$, can be decomposed into finitely many subsequences that are totally ordered by inclusion, from which it results that $(R_k)$ can be decomposed into finitely many almost nested subsequences.
\end{proof}

Let us now state some weak type inequality one can obtain in $L^p$ with our rectangular differentiation processes, using correct factors.


\section{A ``corrected'' weak-type inequality in $L^{p}$}\label{sec.Lp}

As before, we fix a sequence $\bR=(R_k)$ of rectangles in the plane centered at the origin, the diameters of which tend to zero, and we associate to it the differentiation process $\bT:=\bT_{\bR}$ as above.

Recall that one defines the correct factors $Q_k$, $k\in\N$ associated to $\bR$ as in \eqref{eq.corr-fact}, p.~\pageref{eq.corr-fact}, and than one defines, for $1\leq p<\infty$, the $L^p$ correct factors $W_{k,p}$, $k\in\N$ by letting $W_{k,1}:=Q_k$ for $k\in\N$, and, in case one has $1<p<\infty$, by letting for $k\in\N$:
\[W_{k,p} : = {\left (Q_{k} \right)}^{\frac{1}{p}} {\vert R_{k} \vert}^{\frac{1}{p'}},\]
where $1<p'<\infty$ is the conjugate exponent to $p$ satisfying the equality $\frac 1p + \frac{1}{p'}=1$.

Assuming that one has $Q_k<+\infty$ for all $k\in\N$ and given $1\leq p<\infty$, we also define a corrected maximal operator $T^*_p$ by letting, for a locally integrable function $f$ on $\R^2$ and $x \in {\mathbb R}^{2}$:
\[T^*_p f(x):= \sup_{k} \left|\frac{1}{W_{k,p}} \int_{x + R_{k}}  f\right|.\]

The following theorem is a weak type $(p,p)$ inequality for the corrected maximal operator $T^*_p$; it is an adaptation to the continuous case of the discrete weak $\ell^p$ corrected inequality suggested by Rosenblatt and Wierdl in \cite[Comments and problems, p.~551]{RW}.
\begin{Theorem} \label{RWp}
Assume that one has $Q_k<+\infty$ for all $k\in\N$. Then  for 
all $f \in L^{p}({\mathbb R}^{2})$ and all $\lambda >0$, one has 
\[\vert \{x \in {\mathbb R}^{2}: T^*_p f(x) > \lambda \}\vert \leq \frac{1}{\lambda^p} \|f\|_p^p.\]
\end{Theorem}

\begin{proof}

Fix $N  \in {\mathbb N}$. Let $$A_{N} := 
\left\{ x \in {\mathbb R}^{2}: \frac{1}{W_{k,p}} \int_{x + R_{k}} \vert f \vert  > \lambda \text{ for some } 0 \leq k \leq N\right\}.$$ It is clearly sufficient to show that one has $\vert A_{N} \vert \leq  \frac{1}{\lambda^p} \|f\|_p^p$. 
Let $$
E_{N} :=\left\{0\leq k\leq N: \frac{1}{W_{k,p}} \int_{x + R_{k}} \vert f \vert  > \lambda \text{ for some } x \in A_{N}\right\}.
$$ 
For the sake of notational simplicity, we write $A$ and $E$ instead of  $A_{N}$ and $E_N$ respectively.

We now construct two finite sequences of sets $\{A_{k} \}$ and $\{E_{k}\}$. 
Let $A_{0} = A$ and $E_{0} = E$. Let $k_{0} = \min E_{0}$ and choose $x_{0}  \in A_{0}$ such that one has:
\begin{equation}\label{eq.I}
\frac{1}{W_{k_{1},p}} \int_{x_{0} + R_{k_{0}}} \vert f \vert  > \lambda
\end{equation}
Define:
\[B_{0} := x_{0} + \bigcup_{l=k_0}^\infty (R_{k_0}-R_l),\]
so that in particular one has $|B_0|=Q_{k_0}$, and:

\begin{multline*}C_{0}:= \bigg\{x: x \in A_{0}, \frac{1}{W_{i,p}} \int_{x + R_{i}} \vert f \vert  > \lambda\\ \text{and } (x + R_{i}) \cap (x_{1} + R_{k_{1}}) \not= \emptyset \text{ for some } 0 \leq i \leq N\bigg\}.\end{multline*}

Observe that one has $C_0\subseteq B_0$, which follows from the fact that having $x \in C_{0}$ implies that one has $x \in A_{0}$ and $\left( \left( x - x_{0} \right)  + R_{i} \right)  \cap R_{k_{0}} \not= \emptyset$ for some $0 \leq i \leq N$.
Notice also that by the minimality of $k_{0}$, one has $i \geq k_{0}$.

Now by definition of $W_{k_{0}, p}$, one has
\[W_{k_{0}, p} = {\left (Q_{k_{0}} \right)}^{\frac{1}{p}} {\vert R_{k_{0}} \vert}^{\frac{1}{p'}}.\]
Hence, one can write \eqref{eq.I} as 
\[ \int_{x_{1} + R_{k_{0}}} \vert f \vert  > \lambda {\left (Q_{k_{0}} \right)}^{\frac{1}{p}} {\vert R_{k_{0}} \vert}^{\frac{1}{p'}}.\]
By applying H\"older inequality, one thus obtains:
\[{\vert R_{k_{0}} \vert}^{\frac{1}{p'}} {\left ( \int_{x_{1} + R_{k_{0}}} {\vert f \vert}^{p} \right)}^{\frac{1}{p}} \geq \int_{x_{1} + R_{k_{0}}} \vert f \vert  > \lambda {\left (Q_{k_{0}} \right)}^{\frac{1}{p}} {\vert R_{k_{0}} \vert}^{\frac{1}{p'}}, \]
so that
\[   \int_{x_{1} + R_{k_{0}}} {\vert f \vert}^{p}  > {\lambda}^{p} Q_{k_{0}} = {\lambda}^{p} \vert B_{0} \vert. \]

Define now $A_{1} : = A_{0} \setminus B_{0}$ and:
\[E_{1}:= \left\{0 \leq k \leq N: \frac{1}{W_{k,p}} \int_{x + R_{k}} \vert f \vert  > \lambda \text{ for some } x \in A_{1}\right\}.\]
If one has $\vert A_{1} \vert =0$, one stops the procedure. Otherwise, we define $k_{1} = \min E_{1}$ and pick $x_{1} \in A_{1}$. 
One then gets:
\begin{equation}\label{eq.II}
\int_{x_{2} + R_{k_{1}}} \vert f \vert  > \lambda W_{k_{1},p}.
\end{equation}
Define also: $$B_{1} : = x_{1} + \bigcup_{l=k_1}^\infty (R_{k_1}-R_l),$$so that in particular one has $|B_1|=Q_{k_1}$,  and:
\begin{multline*}
C_{1}: = \bigg\{x: x \in A_{1}, \frac{1}{W_{i,p}} \int_{x + R_{i}} \vert f \vert  > \lambda\\
 \text{and } (x + R_{i}) \cap (x_{1} + R_{k_{1}}) \not= \emptyset \text{ for some }i, 0 \leq i \leq N\bigg\}.
 \end{multline*}
Observe that 
\[C_{1} \subseteq B_{1}.\]
Notice that by the minimality of $k_{1}$, one has $i \geq k_{1}$.

By definition of the $L^p$ correct factor, \eqref{eq.II} now rewrites:
\[ \int_{x_{2} + R_{k_{1}}} \vert f \vert  > \lambda {\left (Q_{k_{1}} \right)}^{\frac{1}{p}} {\vert R_{k_{1}} \vert}^{\frac{1}{p'}}.\]
Applying H\"older inequality,   one obtains: 
\[{\vert R_{k_{1}} \vert}^{\frac{1}{p'}} {\left ( \int_{x_{2} + R_{k_{1}}} {\vert f \vert}^{p} \right)}^{\frac{1}{p}} \geq \int_{x_{2} + R_{k_{1}}} \vert f \vert  > \lambda {\left (Q_{k_{1}} \right)}^{\frac{1}{p}} {\vert R_{k_{1}} \vert}^{\frac{1}{p'}}, \]
and hence:

\[   \int_{x_{2} + R_{k_{1}}} {\vert f \vert}^{p}  > {\lambda}^{p} Q_{k_{1}} = {\lambda}^{p} \vert B_{1} \vert. \]
Notice again that, by the definition of $C_{0}$ and using the fact that one has $x_{2} \notin C_{0}$, one can write:
\[\left( x_{1} + R_{k_{0}} \right) \cap \left (x_{2} + R_{k_{1}} \right) = \emptyset.\]

Now we continue this process defining for $i\geq 1$, sets $A_{i} : = A_{i-1} \setminus B_{i-1}$ and:
and 
\[E_{i} :=  \left \{0 \leq k \leq N: \frac{1}{W_{k,p}} \int_{x + R_{k}} \vert f \vert  > \lambda \text{ for some } x \in A_{i}\right\},\]
as long as one has $|A_{i}|>0$.

We claim that after finitely many steps, say $M$ steps, this procedure comes to an end. In other words we assert that there is a $M\in\N$ so that, up to a null set,
\[A \subseteq \bigcup_{i=1}^{M} B_{i}\ ;\]
moreover such an $M$ satisfies the following estimate:
\begin{equation}\label{eq.III}
M\leq {\left (\frac{{\Vert f \Vert}_{p}}{{\alpha}^{\frac{1}{p}} \lambda}\right)}^{p},
\end{equation}
where $\alpha$ is defined by $\alpha := \min_{0 \leq k \leq N} Q_{k}$.

To prove this, assume we are at step $M\in\N$ and that the construction has not yet come to an end. We then have, for every $1 \leq i \leq M$:  
\[   \int_{x_{i} + R_{k_{i}}} {\vert f \vert}^{p}  >  {\lambda}^{p} \vert B_{i} \vert  \geq {\lambda}^{p} \alpha,\]
and, for every $1 \leq i \leq M$ and $1 \leq j \leq M$, with $i \not=j$:
\[\left( x_{i} + R_{k_{i}} \right) \cap \left (x_{j} + R_{k_{j}} \right) = \emptyset.\]
Therefore, one obtains:
\[M\alpha {\lambda}^{p}   \leq  {\lambda}^{p}  \sum_{i=1}^{M}\vert B_{i} \vert   <   \sum_{i=1}^{M}  \int_{x_{i} + R_{k_{i}}} {\vert f \vert}^{p} =  \int_{\bigcup_{i=1}^{M} \left (x_{i} 
+ R_{k_{i}}\right)} {\vert f \vert}^{p} \leq \int_{{\mathbb R}^{2}} {\vert f \vert}^{p} =  {\Vert f \Vert}_{p}^{p}, \]
concluding the proof of \eqref{eq.III}.

Now let $M\in\N$ satisfy \eqref{eq.III} and, up to a negligible set:
\[A \subseteq \bigcup_{i=1}^{M} B_{i}.\]
Recalling that one has $|B_i|=Q_{k_i}$ for each $i$, we hence get:
\[\vert A \vert \leq \sum_{i=1}^{M}\vert B_{i} \vert \leq M\max_{1\leq i\leq M}Q_{k_i}\leq  {\frac{1}{\lambda}{\Vert f \Vert}_{p}^p},\]
and the proof is complete. 
\end{proof}
\begin{Remark}
{It is clear, looking at the proof of the above theorem, that at no place it is important to work with sequences of \emph{rectangles} in the \emph{plane}: $\bR$ could be replaced, in its statement and proof, by any sequence of \emph{bounded Lebesgue-measurable subsets of $\R^d$ having strictly positive Lebesgue measure} and finite correct factors (the definition of which is an immediate generalization of that with rectangles).}
\end{Remark}

\begin{Remark}\label{rmq.equiv}
Clearly, the following two properties are equivalent, by definition of $W_{k,p}$:
\begin{itemize}
\item[(i)] $(Q_k)$ has linear growth (\emph{i.e.} there exists $C>0$ such that one has $Q_k \leq C |R_k|$ for each $k\in\N$);
\item[(ii)] $(W_{k,p})$ has linear growth (\emph{i.e.} there exists $C>0$ such that one has $|W_{k,p}|\leq C |R_k|$ for each $k\in\N$).
\end{itemize}
If one of the two equivalent conditions above is satisfied, then it follows from Theorem \ref{thm.cond-suff} that the associated process $\bT_{\bR}$ is $L^1$-good which, by Lemma \ref{lem.linear}, is equivalent to $\bR$ being almost nested. 
\end{Remark}

The following is an example of a sequence $\bR$ for which $\bT_{\bR}$ is $L^p$-good for all $1<p<\infty$, \emph{without} its correct factors $W_{k,p}$, $k\in\N$ having linear growth. 


\begin{Example}\label{ex.lpgood}
Let $(\lambda_k)$ be a sequence of positive real numbers tending to $0$.


For any $k\in\N$ and $0\leq i\leq k$, let:
$$
R^k_i:=[-2^{-i-1}\lambda_k,2^{-i-1}\lambda_k]\times [-2^{i-k-1}\lambda_k,2^{i-k-1}\lambda_k],
$$
and observe that $\calR_k:=\{R^k_i:0\leq i\leq k\}$ is a collection of $k+1$ rectangles parallel to the coordinate axes, centered at the origin, having the same area $2^{-k}\lambda_k^2$.

Now define a sequence of integers $(n_k)$ by $n_0:=0$ and by requiring that one has $n_k=n_{k-1}+k+1$ for any $k\in\N^*$. One then defines a sequence $\bR:=(R_n)$ by letting, for $k\in\N$ and $n_k\leq n\leq n_k+k$:
$$
R_n:=R^k_{n-n_k}.
$$
This is just ordering the rectangles in $\bigcup_{k\in\N}\calR_k$ by enumerating first rectangles in $\calR_0$, then those in $\calR_1$ \emph{etc.}, and doing so in a way that in each $\calR_k$, $k\in\N$, rectangles are enumerated in decreasing order of their horizontal side-lengths.

Assume now that $\bR$ is decomposed into finitely many (say $N$) subsequences. At least one of them (call it $(R_{m_l})_{l\in\N}$) has the property that for infinitely many indices $k_0<k_1<\cdots<k_j<\cdots$ ($j\in\N$), the set $$
L_j=\{l\in\N : n_{k_j}\leq m_l\leq n_{k_j}+k_j\}
$$
contains at least $k_j/N$ elements.

Given $j\in\N$, write $\min L_j=m_{l_j}$ for $l_j\in\N$ and define $R_j^-:=\min L_j-n_{k_j}$ and $R_j^+:=\max L_j-n_{k_j}$; we know that $R_j^+-R_j^-\geq \#L_j\geq k_j/N$. Yet we compute, for $j\in\N$ (denoting by $Q_{l_j}$ the $l_j^{\text{th}}$ correct factor of the sequence $(R_{m_l})_{l\in\N}$):
\begin{multline*}
Q_{l_j}\geq |R_{n_{k_j}+R_j^-}-R_{n_{k_j}+R_j^+}|=|R^{k_j}_{R_j^-}-R^{k_j}_{R_j^+}|\\
=(2^{-R_j^-}+2^{-R_j^+})(2^{R_j^-}+2^{R_j^+}) 2^{-k_j}\lambda_{k_j}^2\geq 2^{k_j/N} |R_{l_j}|.
\end{multline*}
Hence $(Q_l)$ does \emph{not} have linear growth.

In summary, we can say that \emph{there is no way to decompose $\bR$ into finitely many subsequences, the correct factors of which have linear growth}.
\end{Example}

\begin{Remark}
Example~\ref{ex.lpgood} shows that a sequence $\bR$ can be $L^p$-good for all $1<p<\infty$, without its correct factors $W_{k,p}$, $k\in\N$ having linear growth (the sequence in Example~\ref{ex.lpgood} has all its terms parallel to the coordinate axes).
On the other hand, this is reasonable since in fact the linear growth on $(W_{k,p})_{k\in\N}$ is a condition actually independent of $p$ (it just rephrases the linear growth condition on $(Q_k)$, as was observed in Remark~\ref{rmq.equiv}).
\end{Remark}

Let us now start to develop the tools that will lead us to prove Theorem~\ref{THM2}; among them, the first objects we shall need are (parts of) generalized Perron trees.

\section{Dyadic blocks in generalized Perron trees}\label{sec.Linfini1}

In order to prove Theorem \ref{THM2}, we shall need to exploit, in quite a subtle fashion, (blocks of) generalized Perron trees introduced by Hare and Rönning \cite{HR}.
We here recall their construction and properties which will be crucial to our purposes; for the sake of clarity, we first detail the initial simple geometric construction Perron trees are based upon, and then their recursive construction.

\subsection*{The initial construction}\label{par.in-con}
A basic observation from Hare and Rönning \cite[pp.~217-8]{HR} can be formulated as follows: let $\Delta_1$ and $\Delta_2$ be two adjacent triangles lying from left to right on the horizontal axis, denote their respective side-lengths along the horizontal axis by $x$ and $y$, and their common vertical height-length by $h$ (see Figure~\ref{T1T2}).
\begin{figure}[h]
\begin{center}
\includegraphics[scale=0.9]{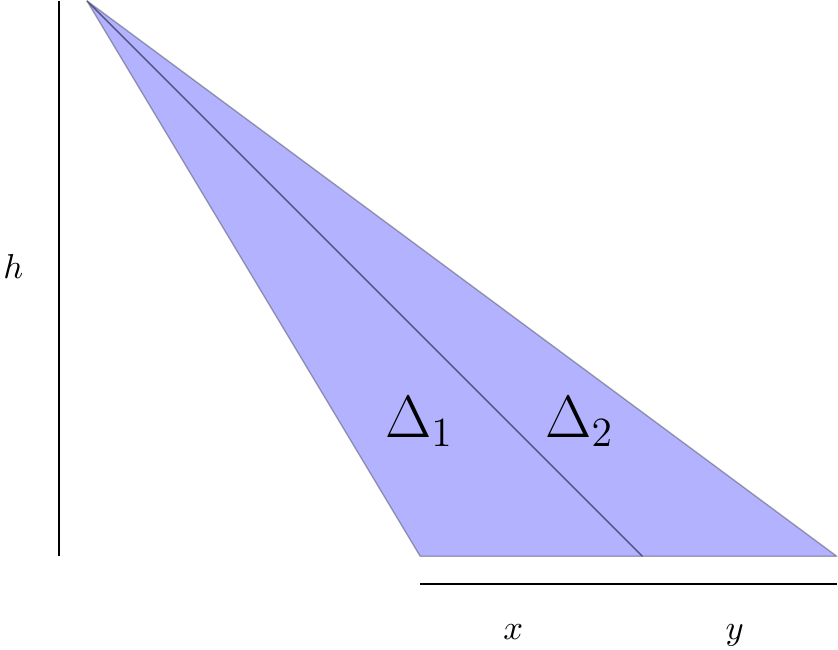}\caption{The triangles $\Delta_1$ and $\Delta_2$}\label{T1T2}
\end{center}
\end{figure}

\begin{figure}[h]
\begin{center}
\includegraphics[scale=0.9]{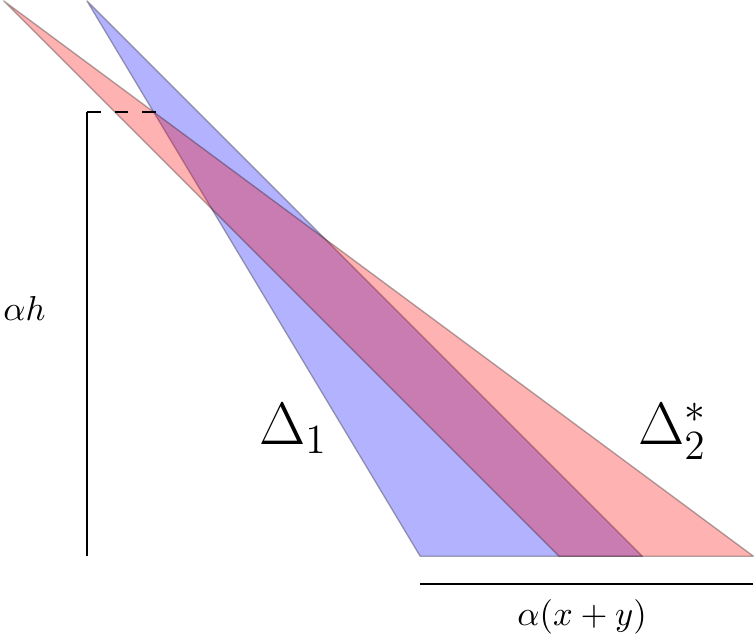}\caption{The triangles $\Delta_1$ and $\Delta_2^*$}\label{T1T2*}
\end{center}
\end{figure}

Let $0<\alpha<1$ be a real number. Denote by $\Delta_2^*$ the triangle obtained by translating horizontally $\Delta_2$ to the left in such a way that its ``right'' side meets the ``left'' side of $\Delta_1$ at a point of height $\alpha h$ (see Figure~\ref{T1T2*}). It is easy to see that $\Delta_1\cup \Delta_2^*$ is the union of a triangle $\Delta^\alpha_0 (\Delta_1,\Delta_2)$, which is the image of the triangle $\Delta_1\cup \Delta_2$ by a similarity of ratio $\alpha$, and of two ``excess triangles'' $\Delta_i^\alpha(\Delta_1,\Delta_2)$, $i=1,2$ whose areas can be easily computed (see Figure~\ref{SDelta}).

\begin{figure}[h]
\begin{center}
\includegraphics[scale=0.9]{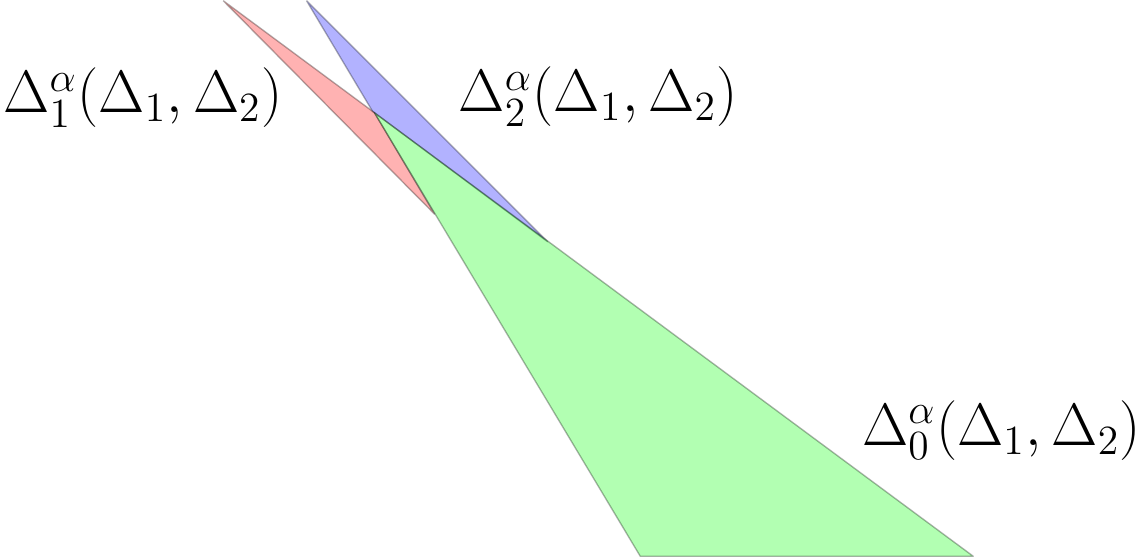}\caption{The triangles $\Delta^\alpha_i(\Delta_1,\Delta_2)$, $i=0,1,2$}\label{SDelta}
\end{center}
\end{figure}

Hence one obtains the following estimate on the area of $\Delta_1\cup \Delta_2^*$:
$$
|\Delta_1\cup \Delta_2^*|=\left[ \alpha^2+(1-\alpha)^2\left(\frac xy + \frac yx \right)\right] |\Delta_1\cup \Delta_2|.
$$

\subsection*{Construction of Perron trees}\label{par.con-pt}
Fix $\bb:=(b_k)_{k\in\N}$ an increasing sequence of nonnegative real numbers with $b_0=0$. Denote, for $k\in\N$, by $\Delta_k$ the triangle in the plane having vertices $A:=(0,1)$, $B_{k-1}:=(b_{k-1}, 0)$ and $B_k:=(b_k,0)$. 

Fix now $n\in\N$ and take a ``dyadic block'' of such triangles $\Delta_{2^n}, \Delta_{2^{n}+1},\dots, \Delta_{2^{n+1}-1}$. The idea is to apply the initial construction described in the previous section to each pair of triangles of the form $(\Delta_{2^n+2j},\Delta_{2^n+{2j+1}})$, $0\leq j\leq 2^{n-1}-1$. Denoting by $\tau_0$ the identity map and by $\tau_j$, $1\leq j\leq 2^{n-1}-1$, translations along the horizontal axis such that the triangles $\Xi_j:=\tau_j \Delta^\alpha_0(\Delta_{2^n+2j},\Delta_{2^n+{2j+1}})$, $0\leq j\leq 2^{n-1}-1$ form a sequence of successively adjacent triangles and defining, for $0\leq j\leq 2^{n-1}-1$, $\Delta_{2^n+2j}':=\tau_j \Delta_{2^n+2j}$ and $\Delta_{2^n+2j+1}':=\tau_j \Delta_{2^n+2j+1}^*$,we get:
$$
\left|\bigcup_{k=2^n}^{2^{n+1}-1} \Delta_k'\right|\leq \left[ \alpha^2+G_n(1-\alpha)^2\right]\left|\bigcup_{k=2^n}^{2^{n+1}-1} \Delta_k\right|,
$$
where one defined:
$$
G_n:=\sup_{0\leq j\leq 2^{n-1}-1} \left(\frac{x_j}{y_j}+\frac{y_j}{x_j}\right),
$$
using the notations $x_j$ and $y_j$ for $b_{2^{n}+2j+1}-b_{2^n+2j}$ and $b_{2^n+2j+2}-b_{2^n+2j+1}$ respectively.

The idea now is to re-start the initial construction on pairs $(\Xi_j,\Xi_{j+1})$, $0\leq j\leq 2^{n-1}-1$ in order to translate horizontally all initial triangles $\Delta_{2^n},\dots, \Delta_{2^{n+1}-1}$ onto new positions $\Delta_{2^n}'',\dots, \Delta_{2^{n+1}-1}''$ in such a way that the area of $\bigcup_{k=2^n}^{2^{n+1}-1} \Delta_k''$ is an (even smaller) proportion of that of $\bigcup_{k=2^n}^{2^{n+1}-1} \Delta_k$, and to repeat this procedure $n$ times in total. This, of course, requires to ensure one has a uniform control on the quantities $\frac xy + \frac yx$ that appear along the procedure.

The following statement gives a condition on the sequence $\bb$ in such a way that the above procedure can be conducted so as to provide translations sending triangles of the $n^{\text{th}}$ dyadic block onto new ones in a way that reduces the area of their union by a factor tending to zero when $n$ tends to infinity. Its proof will be easily supplied by adapting arguments found in Hare and Rönning \cite{HR}.
\begin{Lemma}[Dyadic blocks in generalized Perron trees]\label{lem.pt}
Assume that the sequence $\bb:=(b_k)$ is as above and satisfies:
\begin{equation}\label{eq.cond-bk}
G_{\bb}:=\sup_{n\in\N}\sup_{1\leq l\leq n} \left(\frac{b_{n+2l}-b_{n+l}}{b_{n+l}-b_n}+\frac{b_{n+l}-b_n}{b_{n+2l}-b_{n+l}}\right)<+\infty.
\end{equation}
Denote, for $k\in\N^*$, by $\Delta_k$ the triangle in the plane having vertices $A:=(0,1)$, $B_{k-1}:=(b_{k-1}, 0)$ and $B_k:=(b_k,0)$.

Under those assumptions, there exists a sequence $(\epsilon_n)$ of positive real numbers tending to $0$ and horizontal translations $\tau_k$, $k\in\N$ in such a way that, for any $n\in\N$, one has:
$$
\left| \bigcup_{k=2^n}^{2^{n+1}-1} \tau_k \Delta_k\right|\leq \epsilon_n \left| \bigcup_{k=2^n}^{2^{n+1}-1} \Delta_k\right|.
$$
\end{Lemma}

A second important property we will need in order to prove Theorem~\ref{THM2} can be seen as a simple geometric estimate concerning intersections of a triangle with translates of a rectangle containing it.

\section{A simple geometric observation}\label{sec.simple-geom}

In this section we fix real numbers $0\leq b<c$ and define points $A:=(0,1)$, $B:=(b,0)$ and $C:=(c,0)$ and let $\Delta$ denote the (full) triangle $ABC$. Let $\alpha>0$ be such that one has $AB\geq \alpha BC$.

Denote then by $B'$ and $C'$ the points defined by $\overrightarrow{AB'}=\frac 32 \overrightarrow{AB}$ and $\overrightarrow{AC'}=\frac 32 \overrightarrow{AC}$ and let $\tilde{P}$ be the smallest rectangle having $AB'$ as one of its sides, and containing  the triangle $AB'C'$ (see figure~\ref{fig.Rtilde}).
\begin{figure}[h]
\begin{center}
\includegraphics[scale=1]{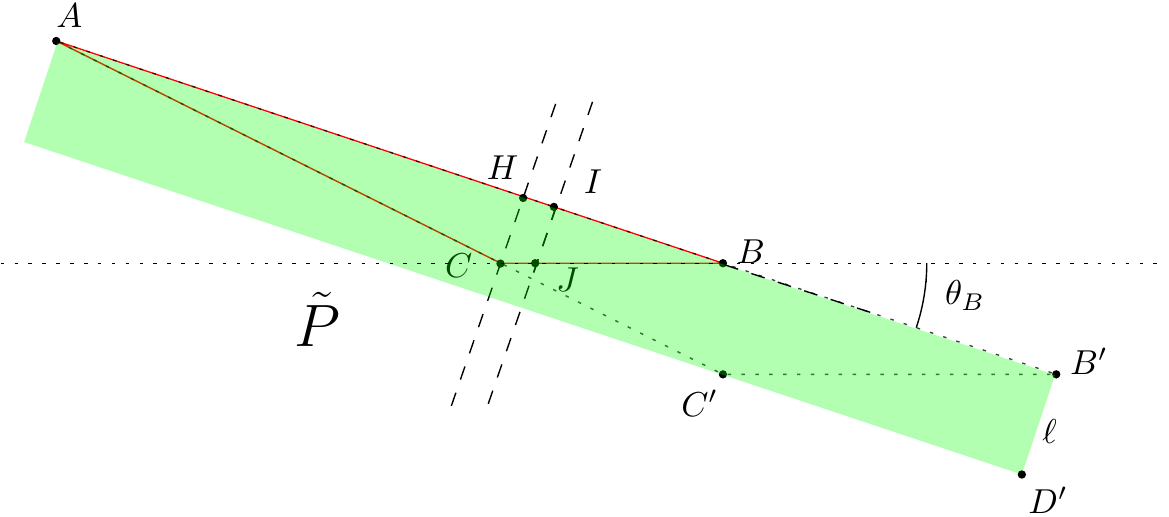}\caption{The triangles $ABC$, $AB'C'$ and the rectangle $\tilde{P}$}\label{fig.Rtilde}
\end{center}
\end{figure}
Finally define $P$ to be the rectangle obtained from $\tilde{P}$ by translating it in such a way that its center is the origin, and let $V$ denote the trapezium $BCC'B'$.

The following estimate, although simple, will be crucial in the sequel.
\begin{Lemma}
For any $x \in V$ one has $$\left|\left(x+P\right) \cap \Delta \right| \geq \frac{\min(\alpha,1)}{72}\left|P\right|.$$
\end{Lemma}
\begin{proof}
Denote by $I$ the middle of the segment $AB'$, by $H$ the point on the segment $AB$ having the property that $CH$ is orthogonal to $AB$ and by $J$ the intersection of $BC$ with the line orthogonal to $AB$ and passing through $I$ (see figure~\ref{fig.Rtilde} again). Finally denote by $D'$ the ``lower right'' vertex of $\tilde{P}$ and let $\ell$ denote the length of the segment $B'D'$.

If $\theta_B\in (0,\frac{\pi}{2})$ stands for the (non-oriented) angle between the horizontal line and $AB$, one obviously has $\tan\theta_B=\frac{1}{b}$.
It is also plain to see, using the triangle $B'C'D'$, that one has: $$\ell=B'D'=B'C'\sin\theta_B=\frac{3}{2} BC\cdot\frac{1}{\sqrt{1+\cot^2\theta_B}}=\frac 32 \cdot\frac{BC}{\sqrt{1+b^2}}=\frac{3}{2}\cdot \frac{BC}{AB}.
$$
Using the triangle $BHC$ in a similar fashion, one gets $CH=BC\sin\theta_B$ and $BH=BC\cos\theta_B$. Since one easily gets $BI=\frac 14 AB$, there comes:
$$
IJ=CH\cdot \frac{BI}{BH}=\frac 14 AB\tan\theta_B=\frac{AB}{4b}.
$$
We hence finally compute:
$$
\frac{IJ}{\ell}=\frac{1}{6b}\cdot\frac{(AB)^2}{BC}=\frac{1}{6}\cdot \frac{AB}{b}\cdot \frac{AB}{BC}=\frac 16 \cdot \frac{\sqrt{1+b^2}}{b}\cdot \frac{AB}{BC}\geq \frac 16 \alpha,
$$
using the assumption we made that $AB\geq\alpha BC$.

The previous observations ensure that the intersection of the rectangle $B'+P$ with the triangle $\Delta=ABC$ contains the triangle formed by the points $B$, $I$ and $J':=I+\frac 16 \alpha\overrightarrow{B'D'}$; hence one has, using the obvious equality $BI=\frac 14 AB=\frac 16 AB'$:
$$
|(B'+P)\cap \Delta|\geq \frac 12 BI\cdot \frac{1}{6}\alpha\ell=\frac{1}{72}\alpha AB'\cdot \ell=\frac{1}{72}\alpha |P|.
$$

Denoting now $K$, $L$ and $M$ respectively the intersections of the ``upper'' side of $C'+P$ parallel to $AB'$ with the lines $AC$, $BC$ and $B'C'$ respectively, one computes:
$$
BL=B'M=\frac 12 B'C'=\frac 34 BC.
$$

It hence follows that $CL=\frac 14 BC$ and that the intersection of $\Delta=ABC$ with $C'+P$, which is the triangle $CKL$, is a triangle homothetic to $ABC$ with ratio $\frac 14$, and thus similar to $AB'C'$ with ratio $\frac 16$, yielding in turn:
$$
|(C'+P)\cap \Delta|=\frac{1}{36} |AB'C'|=\frac{1}{72}|P|.
$$

Now denote by $V$ the trapezion $BCC'B$; it is clear that one has:
\begin{equation}\label{eq.inf-x-V}
\inf_{x\in V} |(x+P)\cap \Delta|\geq \min\{|(B'+P)\cap \Delta|,|(C'+P)\cap \Delta|\}\geq \frac{1}{72}\min(\alpha,1)|P|.
\end{equation}
This completes the proof of the lemma.

\end{proof}

We are now ready to prove Theorem~\ref{THM2}.

\section{Proof of Theorem \ref{THM2}}\label{sec.Linfini2}

Fix, as before, a nondecreasing sequence $\bb=(b_k)$ of positive real numbers satisfying $b_0=0$ and define points $B_k:=(b_k,0)$ for all $k\in\N$ and let $A:=(0,1)$; for each $k\in\N^*$ we then denote by $\Delta_k$ the triangle $AB_{k-1}B_{k}$. For each $k\in\N$, we now associate to $\Delta_k$ a rectangle $P_k$ centered at the origin according to the procedure described in the previous section~\ref{sec.simple-geom}.

Fix also a sequence $\boldsymbol{\delta}=(\delta_k)_{k\in\N^*}$ of positive real numbers and consider the sequence of rectangles $\bR=\bR(\bb,\bdelta)=(R_k)_{k\in\N^*}$ defined by $R_k:=\delta_n P_k$ if $k\in\N^*$ satisfies $2^n\leq k<2^{n+1}$ for a given $n\in\N$; remember that we called $\bdelta$ an \emph{admissible sequence} in case one has $\diam R_k\to 0$, $k\to\infty$.

Let us now (re-)state and prove the main theorem of our paper, referred to before as Theorem~\ref{THM2}, which we recall now.
\begin{Theorem}\label{THM2-v2}
Assume that the sequence $\boldsymbol{b}$ satisfies the following two conditions:
\begin{enumerate}
\item[(i)] there exists a constant $c>0$ such that one has $1 + b_{k-1}^2 \geq c(b_{k} -b_{k-1})^2$ for all $k\in\N^*$;
\item[(ii)] one has: $$G_{ \boldsymbol{b} }:=\sup_{n\in\N}\sup_{1\leq l\leq n} \left(\frac{b_{n+2l}-b_{n+l}}{b_{n+l}-b_n}+\frac{b_{n+l}-b_n}{b_{n+2l}-b_{n+l}}\right)<+\infty.$$
\end{enumerate}
Under those assumptions, the process $\bT_{\bR}$ associated to $\bR = \bR(\bb,\boldsymbol{\delta})$ is $L^\infty$-bad for any admissible sequence $\bdelta$.
\end{Theorem}
\begin{proof}
We keep the notations introduced just above and let $(\epsilon_n)_{n\in\N}$ and $(\tau_k)_{k\in\N^*}$ be associated to $\bb$ (and the corresponding sequence of triangles $\Delta_1$, $\Delta_2$,$\dots$) according to Lemma~\ref{lem.pt}. For $k\in\N^*$, denote also by $V_k$ the trapezium associated to $\Delta_k$ according to the procedure described in section~\ref{sec.simple-geom}. Finally define also, for $n\in\N$:
$$ K^n:=\bigcup_{k=2^n}^{2^{n+1}-1} \tau_k \Delta_k\quad\text{and}\quad V^n:= \bigcup_{k=2^n}^{2^{n+1}-1}\tau_k V_k. $$

Since for any $n\in\N$, $V^n$ contains a similar copy of the triangle $\bigcup_{k=2^n}^{2^{n+1}-1} \Delta_k$ with ratio $\frac 13$, it follows one one hand that one has:
\begin{equation}\label{eq.19}
|V^n| \geq\frac 19 |\bigcup_{k=2^n}^{2^{n+1}-1} \Delta_k|.
\end{equation}
On the other hand, we have, for each $k\in\N^*$, using hypothesis (i):
$$
AB_{k-1}=\sqrt{1+b_{k-1}^2}\geq \sqrt{c} (b_k-b_{k-1})=\sqrt{c}B_{k-1}B_k.
$$
Hence the estimates obtained in section~\ref{sec.simple-geom} hold for $\Delta_k$, $P_k$ and $V_k$ with $\alpha=\sqrt{c}$.

Given $n\in\N$, it now follows from the definition of $V^n$ and from the observations of the previous section that, for each $x\in V^n$, there exists $2^n\leq k\leq 2^{n+1}-1$ such that one has $x\in \tau_k V_k$ and hence also:
$$
|(x+P_k)\cap \tau_k \Delta_k|\geq \frac{\min(\alpha,1)}{72}|P_k|;
$$
and hence also, defining $t_0:= \frac{\min(\alpha,1)}{72}$:
$$
\frac{1}{|P_k|} |(x+P_k)\cap K^n|\geq \frac{1}{|P_k|} |(x+P_k)\cap \tau_k\Delta_k|\geq t_0.
$$
Fix now an admissible sequence $\bdelta$ and define $\bR=\bR(\bb,\bdelta)$ as above. For $n\in\N$ and $x\in\delta_n V^n$, there comes $\frac{x}{\delta_n}\in V^n$ so that the latter computations ensure the existence of an integer $2^n\leq k\leq 2^{n+1}-1$ for which one has:
$$
\frac{1}{|R_k|} |(x+R_k)\cap \delta_nK^n|\geq \frac{1}{|P_k|} \left|\left(\frac{x}{\delta_n}+P_k\right)\cap K^n\right|\geq t_0,
$$
which rewrites:
$$
\frac{1}{|R_k|} \int_{x+R_k} \chi_{\delta_n K^n}\geq t_0,
$$
and implies that one has $T^*(\chi_{\delta_n K^n})(x)\geq t_0$, where $\bT=\bT_{\bR}$ stands for the process associated to $\bR=(R_k)_{k\in\N^*}$.
One hence obtains the following inclusion for each $n\in\N$:
$$
\delta_n V^n\subseteq \left\{ T^*\chi_{\delta_n K^n}\geq t_0\right\}.
$$
Using \eqref{eq.19} and Lemma~\ref{lem.pt}, the above inclusion yields in particular, for each $n\in\N$:
$$
\frac{\left|\left\{ T^*\chi_{\delta_n K^n}\geq t_0\right\}\right|}{|\delta_n K^n|}\geq \frac{|\delta_n V^n|}{|\delta_n K^n|}= \frac{|V^n|}{|K^n|}\geq \frac 19 \frac{\left|\bigcup_{k=2^n}^{2^{n+1}-1} \Delta_k\right|}{\left|\bigcup_{k=2^n}^{2^{n+1}-1} \tau_k\Delta_k\right|}\geq \frac{1}{9\epsilon_n}\to\infty.
$$
But then (see equation \eqref{eq.linfini-good}, p.~\pageref{eq.linfini-good}), $\bT$ cannot be $L^\infty$-good, and the proof is complete.
\end{proof}
\begin{Remark}
In the conditions of the statement of Theorem~\ref{THM2-v2}, the process $\bT$, $L^\infty$-bad, is also $L^p$-bad for any $1\leq p<\infty$.
It hence follows from an observation we present in the Appendix (see Theorem~\ref{thm.appendix}) that, for any $1\leq p<\infty$, the set of functions $f\in L^p(\R^2)$ for which one has $\limsupe_{k\to\infty} |T_kf|<+\infty$ on a set of positive Lebesgue measure, is \emph{meager} (or a \emph{first category} subset) in $L^p(\R^2)$.
\end{Remark}

\appendix
\section{Good functions for our bad processes}

As we mentioned before, we will devote this appendix to show that there are few ``good'' functions in $L^p$ for $L^p$-bad processes. Since this fact holds in a much wider generality than the context of rectangular averaging processes we discussed, we formulate it here for convolution processes, since it seemed to us to be worthwile noticing.

More precisely, we prove the following result.
\begin{Theorem}\label{thm.appendix}
Assume that $(\varphi_k)\subseteq L^1_+(\R^d)$ is a collection of nonnegative integrable functions in $\R^d$ and that $1\leq p<\infty$ is a fixed real number.
Associate to $(\varphi_k)$ a process $\bT=(T_k)$ by letting, for a locally integrable $f$, $T_kf:=\varphi_k*f$. If $\bT$ is $L^p$-bad, then the set:
$$
\calG:=\left\{f\in L^p(\R^d):\limsupe_{k\to\infty} |T_kf|<+\infty\text{ on a set of positive Lebesgue measure}\right\}
$$
is meager in $L^p(\R^d)$.
\end{Theorem}
\begin{proof}
Define, for $n\in\N$, a maximal operator $T^*_n$ by letting, for a locally integrable $f$ and $x\in\R^d$:
$$
T^*_nf(x):=\sup_{k\geq n} |\varphi_k*f(x)|.
$$
Observe that our assumption can be reformulated by saying that $T^*:=T_0^*$ fails to satisfy a weak $(p,p)$-inequality.

For each $n\in\N$, define another operator $S_n$ by letting, for $f$ measurable and $x\in\R^d$:
$$
S_nf(x):=\max_{0\leq k< n} |T_k f(x)|.
$$
Since one has, for each $f\in L^p(\R^d)$:
$$
\|T_kf\|_p\leq \|\varphi_k\|_1\|f\|_p,
$$
it follows that any $T_k$ has weak type $(p,p)$ and hence that the same holds for the operator $S_n$ for any $n\in\N$.
Yet one has, for a measurable $f$ and $x\in\R^d$ and any $n\in\N$:
$$
T^*f(x)=\max\left\{S_nf(x),T_n^*f(x)\right\},
$$
and it hence follows that $T_n^*$ fails to satisfy a weak type $(p,p)$ inequality for any $n\in\N$.
It then follows from \cite[Proposition~1, p.~441]{STEINHA} that for each $n\in\N$ there exists $f_n\in L^p_+(\R^d)$ for which one has $T_n^*f_n=+\infty$ a.e. on $\R^d$. We can of course assume, without loss of generality, that one has $\|f_n\|_p=1$.

Fix now $n\in\N$ and a cube of side $1$, $Q\subseteq\R^d$, and denote by $\calL^d\hel Q$ the (outer) Lebesgue measure restricted to $Q$ (\emph{i.e.} defined for a subset $E\subseteq\R^d$ by $\calL^d\hel Q(E):=|Q\cap E|$). Consider for each integer $k\in\N$ the operator:
$$
T_k^Q:L^p(\R^d)\to L^0(\R^d,\calL^d\hel Q), f\mapsto T_kf.
$$
The operator $T_k^Q$ is continuous in measure since one has, for any $\epsilon$ and any sequence $(f_j)\subseteq L^p(\R^d)$ satisfying $\|f_j-f\|_p\to 0$:
\begin{multline*}
\calL^d\hel Q(\{ |T_kf_j-T_kf|>\epsilon\})\leq |\{ |T_kf_j-T_kf|>\epsilon\}|\\
\leq \frac{1}{\epsilon} \|\varphi*(f_j-f)\|_1\leq \frac{1}{\epsilon}\|\varphi_k\|_1\|f_j-f\|_p,
\end{multline*}
implying obviously that one has $\calL^d\hel Q(\{ |T_kf_j-T_kf|>\epsilon\})\to 0$, $j\to\infty$.

Applying Del Junco and Rosenblatt's \cite[Theorem~1.1]{DJR} to $\calB=L^p(\R^d)$, $\mu=\calL^d\hel Q$ and $(T_k)_{k\geq n}$, we find that there is a dense $\calG_\delta$ subset $\calX_{Q,n}\subseteq L^p(\R^d)$ such that one has $T_n^*f=+\infty$,  $(\calL^d\hel Q)$-a.e., that is $T_n^*f=+\infty$ a.e. on $Q$, for any $f\in \calX_{Q,n}$.

Now let:
$$
\calX_n:=\bigcap_{\nu\in\Z^n} \calX_{\nu+Q,n}.
$$
Obviously $\calX_n$ is a dense $\calG_\delta$ in $L^p(\R^d)$; for $f\in \calX_n$ we then have $T_n^*f=+\infty$ a.e. on $\nu+Q$ for any $\nu\in\Z^n$, and hence $T_n^*f=+\infty$ for a.e. $x\in\R^d$.

Finally let $\calX:=\bigcap_{n\in\N} \calX_n$, observe that $\calX$ is a dense $\calG_\delta$ subset of $L^p(\R^d)$ and that for $f\in \calX$ we have, for all $n\in\N$ and a.e. $x\in\R^d$:
$$
T_n^* f(x)=+\infty.
$$
Hence we have, for $f\in\calX$ and a.e. $x\in\R^d$:
$$
\limsupe_{k\to\infty} |T_kf(x)|=\inf_{n\in\N}\sup_{k\geq n}|T_kf(x)|=\inf_{n\in\N} T_n^*f(x)=+\infty.
$$
The result now immediately follows.
\end{proof}

\subsection*{Acknowledgements.} This research has been supported by the \emph{École Normale Supérieure}, Paris, by the \emph{``Gruppo Nazionale per l’Analisi Matematica, la Probabilità e le loro Applicazioni dell’Istituto Nazionale di Alta Matematica F. Severi''} and by the project \emph{Vain-Hopes} within the program \emph{Valere} of \emph{Università degli Studi della Campania ``Luigi Vanvitelli''}. It has also been partially accomplished within the \emph{UMI} Group \emph{TAA ``Approximation Theory and Applications''}.
E.~D’Aniello would like to thank the \emph{École Normale Supérieure}, Paris, for its warm hospitality in March-April, 2022; L.~Moonens would like to thank the \emph{Dipartimento di Matematica e Fisica} of the \emph{Università degli Studi della Campania ``Luigi Vanvitelli''} for its warm hospitality in April-May, 2022.

\bibliographystyle{plain}
\bibliography{DGMRbib}

\begin{thebibliography}{10}

\bibitem{BATEMAN}
Michael Bateman.
\newblock Kakeya sets and directional maximal operators in the plane.
\newblock {\em Duke Math. J.}, 147(1):55--77, 2009.

\bibitem{BURKHOLDER}
Donald~L. Burkholder.
\newblock Maximal inequalities as necessary conditions for almost everywhere
  convergence.
\newblock {\em Z. Wahrscheinlichkeitstheorie und Verw. Gebiete}, 3:75--88
  (1964), 1964.

\bibitem{DM17}
Emma D'Aniello and Laurent Moonens.
\newblock Averaging on {$n$}-dimensional rectangles.
\newblock {\em Ann. Acad. Sci. Fenn. Math.}, 42(1):119--133, 2017.

\bibitem{DMZAA}
Emma D'Aniello and Laurent Moonens.
\newblock Differentiating {O}rlicz spaces with rectangles having fixed shapes
  in a set of directions.
\newblock {\em Z. Anal. Anwend.}, 39(4):461--473, 2020.

\bibitem{DMR}
Emma D'Aniello, Laurent Moonens, and Joseph Rosenblatt.
\newblock Differentiating orlicz spaces with rare bases of rectangles.
\newblock {\em Ann. Acad. Sci. Fenn. Math.}, 45(1):411--427, 2020.

\bibitem{GUZMAN}
Miguel de~Guzm\'{a}n.
\newblock {\em Differentiation of integrals in {$R^{n}$}}.
\newblock Lecture Notes in Mathematics, Vol. 481. Springer-Verlag, Berlin-New
  York, 1975.
\newblock With appendices by Antonio C\'{o}rdoba, and Robert Fefferman, and two
  by Roberto Moriy\'{o}n.

\bibitem{DJR}
Andr\'{e}s del Junco and Joseph Rosenblatt.
\newblock Counterexamples in ergodic theory and number theory.
\newblock {\em Math. Ann.}, 245(3):185--197, 1979.

\bibitem{GARSIA}
Adriano~M. Garsia.
\newblock {\em Topics in almost everywhere convergence}.
\newblock Lectures in Advanced Mathematics, No. 4. Markham Publishing Co.,
  Chicago, Ill., 1970.

\bibitem{GAUVANMEMOIRE}
Anthony Gauvan.
\newblock {\em {Differentiation theorems associated to bases of rectangles}}.
\newblock 2020.
\newblock Master's Thesis~--~Université Paris-Diderot.

\bibitem{GAUVANPT}
Anthony Gauvan.
\newblock {Application of Perron Trees to Geometric Maximal Operators}.
\newblock {\em Submitted, arXiv:2204.00253}, 2022.

\bibitem{HP}
Paul Hagelstein and Ioannis Parissis.
\newblock Tauberian constants associated to centered translation invariant
  density bases.
\newblock {\em Fund. Math.}, 243(2):169--177, 2018.

\bibitem{HR}
Kathryn~E. Hare and Jan-Olav R\"{o}nning.
\newblock Applications of generalized {P}erron trees to maximal functions and
  density bases.
\newblock {\em J. Fourier Anal. Appl.}, 4(2):215--227, 1998.

\bibitem{JMZ}
{Borge} Jessen, {J\'{o}zef} Marcinkiewicz, and Antoni Zygmund.
\newblock Note on the differentiability of multiple integrals.
\newblock {\em Fundam. Math.}, 25:217--234, 1935.

\bibitem{MR}
Laurent Moonens and Joseph~M. Rosenblatt.
\newblock Moving averages in the plane.
\newblock {\em Illinois J. Math.}, 56(3):759--793, 2012.

\bibitem{ONIANI}
Giorgi Oniani.
\newblock On the differentiation of integrals with respect to translation
  invariant convex density bases.
\newblock {\em Fund. Math.}, 246(2):205--216, 2019.

\bibitem{RW}
Joseph~M. Rosenblatt and M\'{a}t\'{e} Wierdl.
\newblock A new maximal inequality and its applications.
\newblock {\em Ergodic Theory Dynam. Systems}, 12(3):509--558, 1992.

\bibitem{SAWYER}
Stanley~A. Sawyer.
\newblock Maximal inequalities of weak type.
\newblock {\em Ann. of Math. (2)}, 84:157--174, 1966.

\bibitem{STEINART}
Elias~M. Stein.
\newblock On limits of seqences of operators.
\newblock {\em Ann. of Math. (2)}, 74:140--170, 1961.

\bibitem{STEINHA}
Elias~M. Stein.
\newblock {\em Harmonic analysis: real-variable methods, orthogonality, and
  oscillatory integrals}, volume~43 of {\em Princeton Mathematical Series}.
\newblock Princeton University Press, Princeton, NJ, 1993.
\newblock With the assistance of Timothy S. Murphy, Monographs in Harmonic
  Analysis, III.

\bibitem{STOKOLOS88}
Alexander~M. Stokolos.
\newblock On the differentiation of integrals of functions from {$L\varphi
  (L)$}.
\newblock {\em Studia Math.}, 88(2):103--120, 1988.

\bibitem{STOKOLOS2005}
Alexander~M. Stokolos.
\newblock Zygmund's program: some partial solutions.
\newblock {\em Ann. Inst. Fourier (Grenoble)}, 55(5):1439--1453, 2005.

\end{thebibliography}

\vspace{0.3cm}
\noindent
\small{\textsc{Emma D'Aniello}}\\[0.1cm]
\small{\textsc{Dipartimento di Matematica e Fisica},}
\small{\textsc{Universit\`a degli Studi della Campania ``Luigi Vanvitelli''},}
\small{\textsc{Viale Lincoln n. 5, 81100 Caserta,}}
\small{\textsc{Italia}}\\
\footnotesize{\texttt{emma.daniello@unicampania.it}.}

\vspace{0.3cm}
\noindent
\small{\textsc{Anthony Gauvan}}\\[0.1cm]
\small{\textsc{Laboratoire de Math\'ematiques d'Orsay, Universit\'e Paris-Sud, CNRS UMR8628, Universit\'e Paris-Saclay},}
\small{\textsc{B\^atiment 307},}
\small{\textsc{F-91405 Orsay Cedex},}
\small{\textsc{France}}\\
\footnotesize{\texttt{anthony.gauvan@universite-paris-saclay.fr}.}

\vspace{0.3cm}
\noindent
\small{\textsc{Laurent Moonens}}\\[0.1cm]
\small{\textsc{Laboratoire de Math\'ematiques d'Orsay, Universit\'e Paris-Saclay, CNRS UMR 8628},}
\small{\textsc{B\^atiment 307},}
\small{\textsc{F-91405 Orsay Cedex},}
\small{\textsc{France}}\\
\footnotesize{\texttt{laurent.moonens@universite-paris-saclay.fr}.}

\noindent\small{\textsc{\'Ecole Normale Supérieure-PSL University, CNRS UMR 8553},}
\small{\textsc{45, rue d'Ulm},}
\small{\textsc{F-75230 Parix Cedex 3},}
\small{\textsc{France}}\\
\footnotesize{\texttt{laurent.moonens@ens.fr}.}

\vspace{0.3cm}
\noindent
\small{\textsc{Joseph M. Rosenblatt}}\\[0.1cm]
\small{\textsc{Department of Mathematics, University of Illinois at Urbana-Champaign},}
\small{\textsc{Urbana, IL 61801},}
\small{\textsc{U.S.A.}}\\
\footnotesize{\texttt{rosnbltt@illinois.edu}.}

\end{document}